\documentclass{article}
\usepackage{placeins}

\title{Trefftz Discontinuous Galerkin approximation of an acoustic waveguide\thanks{The research of the first author was partially supported by the US AFOSR under grant number FA9550-23-1-0256. The research of the second author was partially supported by MICINN/AE/doi 10.13039/501100011033 grants PID2020-114173RB-I00 and PID2023-147790OB-I00 and his stay at the University of Oviedo was funded by a Margarita Salas grant, from the Ministerio de Universidades (RD 289/2021) funded by the European Union-NextGenerationEU.  The research of the third author was partially supported MICINN grant PID2020-116287GB-I00.}}
\author{Peter Monk\thanks{Mathematical Sciences Department, University of Delaware, Newark, DE 19716, USA (email: \texttt{monk@udel.edu}).} \and Manuel Pena\thanks{Depto. Matem\'atica Aplicada I, Universidade de Vigo, Campus das Lagoas-Marcosende, 36310 Vigo, Spain (email: \texttt{manuel.pena@uvigo.gal}).}  \and Virginia Selgas\thanks{Depto. Matem\'aticas, EPIG, Universidad de Oviedo, 33203 Gij\'on, Spain (email: \texttt{selgasvirginia@uniovi.es}).}}

\usepackage{amsmath,amssymb,amsthm}
\usepackage{graphicx}
\usepackage{float}
\usepackage{url,hyperref}

\usepackage{tikz}
\usetikzlibrary{calc,patterns}

\newtheorem{remark}{Remark}
\newtheorem{lemma}{Lemma}
\newtheorem{theorem}{Theorem}

\newcommand{\xbold}{\boldsymbol{x}}
\newcommand{\ybold}{\boldsymbol{y}}

\newcommand{\nbold}{\boldsymbol{n}}
\newcommand{\dbold}{\boldsymbol{d}}
\newcommand{\vbold}{\boldsymbol{v}}

\newcommand{\R}{\mathbb{R}}

\newcommand{\jump}[1]{\left|\! \left[ #1 \right]\!  \right|}
\newcommand{\jumpn}[1]{\left|\! \left[ #1 \right]\!  \right|_{\nbold}}

\newcommand{\av}[1]{\left\{ \!\! \left\{ #1 \right\} \!\!  \right\} }

\begin{document}

\maketitle

\begin{abstract}
We propose a modified Trefftz Discontinuous Galerkin (TDG) method for approximating a time-harmonic acoustic scattering problem in an infinitely elongated waveguide. In the waveguide we suppose there is a bounded, penetrable and possibly absorbing scatterer. The classical TDG is not applicable to this important case. Novel features of our modified TDG method are that it is applicable when the scatterer is absorbing, and it uses  a stable treatment of the asymptotic radiation condition for the scattered field. For the modified TDG, we prove $h$ and $p$-convergence in the $L^2$ norm. The theoretical results are verified numerically for a discretization based on plane waves (that may be evanescent in the scatterer).
\end{abstract}


\section{Introduction}\label{sec:intro}
Trefftz methods~\cite{Trefftz26} are numerical  schemes that use a superposition of known exact solutions to a problem to approximate the true solution. These methods can have very fast convergence when the solution is sufficiently regular (see e.g. \cite{barnett08}).  However in applications to acoustics, elasticity or electromagnetism the exact solution can lose regularity near corners (or edges in 3D) of the solution domain, as well as at interfaces between different materials. In order to accurately approximate  such solutions and  to help
control the conditioning of the overall linear system, it has proven useful to employ Trefftz techniques element-wise on a finite element mesh.  Appropriate inter-element continuity is imposed variationaly, and the resulting scheme is then called a Trefftz Discontinuous Galerkin (TDG) scheme.

For the Helmholtz equation, which is the subject of this paper, an error analysis of TDG  was first performed in \cite{Gittelson+al2009} (see \cite{hmp16} for a survey of TDG methods for the Helmholtz equation, and further references to the literature). In \cite{Gittelson+al2009}, they noted that, when a special choice of parameters is made in TDG, the method is equivalent to another Trefftz method: the Ultra Weak Variational Formulation (UWVF) proposed in \cite{cessenat96,CessenatDespres1998} and analyzed in \cite{bp08}.  But this 
equivalence only holds when there is no loss (or absorption) in the material so that the coefficients in the Helmholtz equation are real-valued. Indeed the standard TDG is not applicable when there is loss since the derivation assumes real coefficients and breaks down for complex coefficients. The TDG method needs to be modified in the presence of loss, and this is one of the main novelties of this paper.

We are interested in  simulating the propagation of time-harmonic waves along a tubular waveguide in order to produce data for testing inverse scattering algorithms.  TDG is an attractive alternative to standard solvers such as the Finite Element Method (FEM) because it is more efficient at solving higher frequency problems, cf.~\cite{Lahivaara24}.  

In order to apply the TDG method to this problem, we first rewrite the 
problem on a bounded computational domain by using the Neumann-to-Dirichlet (NtD) map on the truncation walls. This choice, rather than the more standard Dirichlet-to-Neumann map, is helpful since traces of discrete TDG fields are only in $L^2$ on the boundary. A second contribution of this paper is a variational formulation, related to least squares methods, that incorporates the NtD map in a stable way.

Typically, for FEM, a Perfectly Matched Layer (PML) is used to terminate the bounded section of the waveguide, but the presence of evanescent and travelling waves in the solution complicates the design of the PML. In addition the PML cannot be implemented in the standard TDG since it cannot handle loss because PML parameters are complex valued. In contrast the NtD map  has only one discretization parameter which can be picked a priori in a simple way, and generalizes to arbitrary cross section waveguides.

After modifying the TDG scheme to accept lossy materials, and to use the NtD map, we prove convergence of the resulting scheme via modifications of the techniques in \cite{Gittelson+al2009,hiptmair+al2014}.  In particular, we obtain a  consistent and coercive formulation.  This achieves  convergence with respect to $h$ (mesh size) and $p$ (number of Trefftz functions per element) when discretized with plane waves (under an assumption that the approximation theory of plane waves extends to the Helmholtz equation with complex valued coefficients). We illustrate the convergence theory by numerical experiments in two cases: first an example without loss testing the NtD implementation, and then an example with loss where we compare to a finite element solution.

One potential advantage of the TDG scheme over the UWVF is the presence of parameters that can be tuned to improve the accuracy of the solution or the conditioning of the discrete problem.  In \cite{hiptmair+al2014} TDG methods on highly refined grids were used to approximate the solution near corners in the domain where regularity is lost. In that case mesh dependent parameters were needed to prove error estimates.  
However in \cite{congreve17} it was reported that the effect of these mesh dependent parameters is not significant.  The tubular waveguide used in our study allows for a simple test of mesh dependent parameters in a problem with a smooth solution: we report results for  wave propagation down a simple tube through a  refined mesh region.

We remark that the study  here does not address the ill-conditioning that can appear in TDG methods based on propagating plane waves, cf. \cite{Huttunen2007,congreve19}. {For some novel potential mitigation strategies not considered here, see~\cite{parolin23} and \cite{barucq21}.}

The paper is structured as follows. In Section~\ref{sec:modelprob} we introduce the model problem of acoustic scattering in an infinite waveguide, inside of which there is a penetrable and possibly absorbing scatterer. For this problem, in Section~\ref{sec:varform} we propose a consistent TDG weak formulation. Then, in  Section~\ref{sec:convTrefftz} we analyze this modified TDG formulation both at the continuous and 
discrete level, first in a mesh dependent norm and finally in the $L^2$ norm. We provide  error bounds for the convergence when using a discretization with plane waves. The convergence order is illustrated in  Section~\ref{sec:numerics}, 
where we provide numerical results for wave propagation down a simple tube and compare to solutions from a finite element method for a lossy scatterer. We also investigate mesh 
dependent parameters. Finally, we draw some conclusions in Section~\ref{sec:conclusion}. 


\section{Model problem}\label{sec:modelprob}
We consider an infinite tubular waveguide $\Omega=\R\times\Sigma\subset \R^m$ ($m=2$ or $3$) with a constant cross section $\Sigma\subset\R^{m-1}$ that is bounded and convex, and, for $m=3$ piecewise smooth  in $\R^{m-1}$.  

We remark that the upcoming analysis applies to more exotic configurations in which $\Omega$ has the shape of a tubular waveguide away from a bounded section. That is, it suffices that there exists $R'>0$ such that $\Omega \cap \{ (x_1,\hat{x})\in\R\times\R^{m-1}; \, |x_1|> R'\} = ((-\infty,-R') \cup (R',\infty) )\times \Sigma $, as long as the associated problem is well posed for the given wave number  and the solution is sufficiently smooth. 

We assume that the waveguide is filled with a homogeneous and isotropic material, and that inside this background material there is a penetrable scatterer that occupies a Lipschitz smooth and bounded region $\Omega_i $ which, for simplicity, we assume does not touch the boundary of the waveguide (i.e. $\overline{\Omega_i}\subset \Omega$). When necessary, we identify a point $\xbold$ in the waveguide $\Omega=\R\times\Sigma$ with $(x_1,\hat{\mathbf{x}})$, where $x_1\in \R$ and $\hat{\mathbf{x}}\in\Sigma$. We will also denote by $\Sigma_c=\{c\}\times\Sigma$ the cross section at a given $x_1=c$ (for $c\in\mathbb{R}$) away from the scatterer.

The bounded computational domain in which we apply TDG is a finite section $\Omega_{R}=(-R,R)\times\Sigma$, see Fig.~\ref{fig:crosssection}. We take $R>0$ big enough so that the scatterer is completely contained in such a section, that is, $\overline{\Omega_i}\subset\Omega_{R}$. Moreover, we denote the exterior of the scatterer in $\Omega_{R}$ by $\Omega_{e,R}=\Omega_R\setminus\overline{\Omega_{i}}$ and we split its boundary in the following parts: 
\begin{itemize}
\item the left and right truncation walls $S_R=\Sigma_{-R}\cup\Sigma_{+R}$; 
\item the bounded section  $\Gamma_{R}=(-R,R)\times\partial\Sigma$ of the boundary of the whole waveguide $\Gamma=\partial\Omega = \R\times\partial\Sigma$;
\item and the boundary of the scatterer, $\partial\Omega_i$.
\end{itemize}

The forward problem consists of determining the response of the system when some known incident field $u_{inc}$ strikes the target. We suppose that the incident field solves the Helmholtz equation in $\Omega_R$: 
$$
\Delta  u_{inc} +k^2\, u_{inc}=0\quad \mbox{in }  \Omega_R\, ,
$$
where $k \in \R$, $k>0$,  denotes the wave number in the waveguide; moreover, assuming that the wall of the waveguide is sound-hard, the incident field must also satisfy 
\begin{equation}\label{bcuinc_soundhard}
\partial_{\mathbf{n}_{\Gamma_R} }  u_{inc} =0 \quad \mbox{on } \Gamma_{R}\, ,
\end{equation}
where $\mathbf{n}_{\Gamma_R}$ is the unit normal vector on  $\Gamma_R$ that points out of $\Omega_R$. 
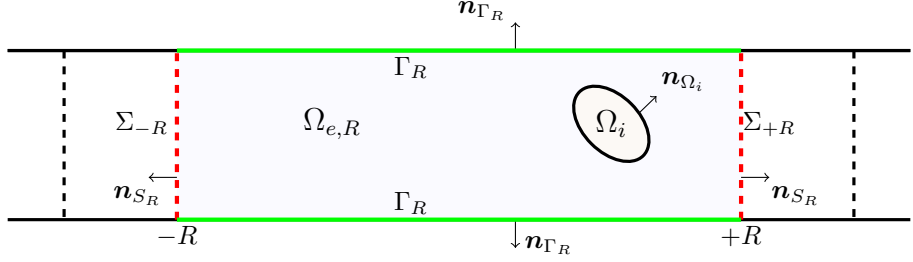
\begin{figure}
   \centering
   \begin{tikzpicture}[scale=0.75]
    \fill[color=blue!2] (-1,-1.5) rectangle (9,1.5);
    \draw [very thick](-4,1.5) -- (12,1.5);
    \draw [very thick](-4,-1.5) -- (12,-1.5);
    \draw [color=red, ultra thick, dashed](9,-1.5) -- (9,1.5);
    \draw [color=red, ultra thick, dashed](-1,-1.5) -- (-1,1.5);
    \filldraw[fill=brown!5, very thick,rotate around={45:(6.7,0.2)}](6.7,0.2) ellipse (0.5 and 0.8);
    \node[font = \large] at (1.75, 0.2){$ {\Omega_{e,R}}$};
    \node[font = \large] at (6.7, 0.2){$\displaystyle\Omega_i$};
    \node[font = \normalsize] at (-1.625,.2){$\displaystyle\Sigma_{-R}$};
    \node[font = \normalsize] at (-1,-1.8){$\displaystyle -R$};
    \node[font = \normalsize] at (9.5,0.2){$\displaystyle\Sigma_{+R}$};
    \node[font = \normalsize] at (9,-1.8){$\displaystyle +R$};
    \draw [->](7.2,0.4) -- (7.5,0.69) node[near end, above right,font = \normalsize] {$\displaystyle \nbold_{\Omega_i}$};
    \draw [->](-1,-0.75) -- (-1.5,-0.75) node[near end, below left,font = \normalsize] {$\displaystyle\nbold_{S_{R\hspace*{-.2cm}} }$};
    \draw [->](9,-0.75) -- (9.5,-0.75) node[near end, below right,font = \normalsize] {$\displaystyle\nbold_{S_{R}}$};
    \draw [->](5,-1.5) -- (5,-2) node[near start, below right,font = \normalsize] {$\displaystyle \nbold_{\Gamma_R}$};
    \draw [->](5,1.5) -- (5,2) node[near end, above left,font = \normalsize] {$\displaystyle \nbold_{\Gamma_R}$};
    \draw [very thick](-4,1.5) -- (12,1.5);
    \draw [color=green, ultra thick](-1,-1.5) -- (9,-1.5);
    \draw [color=green, ultra thick](-1,1.5) -- (9,1.5);
    \node[font = \normalsize ] at (3.15,-1.2){$\displaystyle \Gamma_R$};
    \node[font = \normalsize] at (3.15,1.2){$\displaystyle \Gamma_R$};
    \end{tikzpicture}
    \caption{A sketch of the geometry of the model problem.}
    \label{fig:crosssection}
\end{figure}

We want to determine the response of the system, which we next formalize in terms of the total field $u$ or, equivalently, in terms of the scattered field $u_{sc}=u-u_{inc}$. The total field satisfies the
following source free Helmholtz equation:
\begin{equation}\label{eq:Helm}
    \Delta u +k^{2}\,\mathrm{n} \, u= 0\quad\mbox{in }  \Omega_R \, .
\end{equation}
The scatterer is characterised by the  coefficient function $\mathrm{n}$ defined piecewise by $\mathrm{n}=1$ in $\Omega_{e}$ and $\mathrm{n}=\mathrm{n}_i(\xbold)$ in $\Omega_i$. Since we are using a Trefftz method, we assume that $ \mathrm{n}_i$ is piecewise constant in $\Omega_i$, and that its real part is uniformly positive whereas its imaginary part is non-negative:
$$
\Re (\mathrm{n}_i )\geq \mathrm{n}_0>0 \quad\mbox{and}\quad \Im (\mathrm{n}_i )\geq 0 \quad\mbox{a.e. in } \Omega_i \, .
$$
Physically, $\Re(\mathrm{n}(\xbold))=c_0^2/c(\xbold)^2$ where $c_0$ is the sound speed in the background medium, and $c(\xbold)\not=c_0$ is the sound speed in the scatterer. In addition, if the imaginary part is positive, $\Im(\mathrm{n})>0$, there is loss in the medium.

In order to impose transmission conditions on $\partial\Omega_i$, and because TDG uses discontinuous functions on the grid, we need to define jumps and average values of functions on adjacent elements. Suppose two non-intersecting domains $\mathcal{O}$ and $\mathcal{O}'$ meet on a common surface $\mathcal{S}$, and have unit outward normals  $\nbold_{\mathcal{O}}$ and $\nbold_{\mathcal{O}'}$ respectively.  Then if $f_{\mathcal{O}}\in H^1(\mathcal{O})$ and
$f_{\mathcal{O}'}\in H^1(\mathcal{O}')$ we
can define the average value on ${\cal S}$ by
\[
\av{f}=\frac{f_{\mathcal{O}}+f_{\mathcal{O}'}}{2}\quad\mbox{ on } \mathcal{S}.
\]
The same notation is used for the average value
of a vector quantity. We also define in the standard way (cf.~\cite{arn02a}) the normal jump of scalar and vector fields. In particular, for a scalar field $f$ the normal jump is a vector quantity:
\[
\jump{f}_{\nbold}=f_{\mathcal{O}}\,\nbold_{\mathcal{O}}+f_{\mathcal{O}'}\,\nbold_{\mathcal{O}'}\quad\mbox{on } \mathcal{S};
\]
whereas for a smooth enough vector field $\vbold$ defined in $\mathcal{O}$ and $\mathcal{O}'$ the normal jump is a scalar quantity:
\[
\jump{\vbold}_{\nbold}=\vbold_{\mathcal{O}}\cdot\nbold_{\mathcal{O}}+\vbold_{\mathcal{O}'}\cdot \nbold_{\mathcal{O}'}\quad\mbox{on } \mathcal{S}.
\]
Then we can state the transmission conditions of our model problem as follows:
\begin{equation*}
    \jump{ u}_{\nbold}= \boldsymbol{0} \mbox{ and }
    \jump{\nabla u}_{\nbold} =0 \quad  
    \mbox{ on } \partial\Omega_i \, .
\end{equation*}
We still need to impose suitable {boundary conditions} for our forward problem on the boundary of the computational domain.  For a waveguide with a sound-hard wall we have
$$ 
\partial_{\nbold_{\Gamma_R}\! } u = 0 \quad\mbox{on } \Gamma _R \, .
$$ 

On the artificial truncation walls, we need to impose a radiation  condition on the scattered field $u_{sc}$ that represents that its counterpart in the time domain is a superposition of outgoing guided modes (i.e. guided modes propagating away from the obstacle $\Omega_{i}$) or modes  that are decaying exponentially with respect to the distance from the obstacle $\Omega_{i}$), cf. \cite{bourgeois+luneville2008}.  To this end: 
\begin{itemize}
\item Consider the sequence $\{k_{j}^2\}_{j=0}^{\infty}\subset [0,+\infty)$ of eigenvalues of the Neumann problem for the negative Laplacian on $\Sigma$:
	\begin{equation*}
		\left\{\begin{array}{ll}
		-\Delta_{\hat{\xbold}}\theta_j  =k_j^2\, \theta_j 
		\quad & \mbox{in } \Sigma\\[1ex]
		 \partial_{\hat{\nbold }_{\Sigma}}\theta_j   
		=0 & \mbox{on }\partial\Sigma
		\end{array}\right.
	\end{equation*}
where $\hat{\nbold}_{\Sigma}$ represents the unit vector field normal to $\partial\Sigma$ that points outside of $\Sigma\subset\R^{m-1}$. Such eigenvalues have no finite accumulation point, and can be enumerated in non-decreasing order (including their finite multiplicity) so that $k_0=0$, $k_{j+1}\geq k_j $ for every $j=0,1,\ldots$, and  $k_j\rightarrow  \infty$ when $j\rightarrow\infty$. Moreover, for each  eigenvalue $k_j^2$ we may take an associated eigenfunction $\theta_j $ so that  $\{\theta_j\}_{j=0}^\infty$ is an orthonormal basis of $L^2(\Sigma)$.
\item We can now define the longitudinal wave number  $\beta_j=\sqrt{ k^2-k_j^2}\in\mathbb{C}$ for each $j=0,1,\ldots$ by choosing the branch of the square root for which  $\Im(\beta_j)\geq 0 $, and consider the waveguide modes
    \begin{equation}\label{def:modes} 
        g_j^\pm(x_1,\hat{\xbold})
         \, =\, e^{\pm i\beta_j x_1}\, \theta_j(\hat{\xbold }) \quad\mbox{for }\xbold =(x_1,\hat{\xbold})\in\Omega \, .
    \end{equation}
\end{itemize}
We assume that $k \in\R\setminus\{ k_j\}_{j=0}^\infty$ (i.e not a cutoff wave number) so that $\beta_j\neq 0$ for all $j=0,1,\ldots$. Under this assumption, there exists $N_{\rm{}pr}\in\mathbb{N}$ such that $k_j^2< k^2$ for $j\leq N_{\rm{}pr}$ and $k_j^2>  k^2 $ for $j>N_{\rm{}pr}$. In consequence:
\begin{itemize}
    \item When $j\leq N_{\rm{}pr}$, the waveguide modes $g_j^\pm(x_1,\hat{\xbold}) $ defined  by (\ref{def:modes}) correspond to travelling modes that propagate from left to right ($+$) and from right to left ($-$), respectively. They are sorted according to their longitudinal wave number $\beta_0=k\geq \beta_1\geq \ldots \geq \beta_{N_{\rm{}pr}} >0$.
    \item When $j > N_{\rm{}pr}$, the waveguide modes $g_j^\pm(x_1,\hat{\xbold})$ correspond to {decaying modes} that decrease exponentially from left to right ($+$) and from right to left ($-$), respectively. They are sorted according to their decay rate $    0< |\beta_{N_{\rm{}pr}+1}| \leq |\beta_{N_{\rm{}pr}+2} |\leq \ldots $ 
\end{itemize}

Normally for FEM, the Dirichlet-to-Neumann (DtN) map is used to prove the existence of solutions to the waveguide problem (except -at most- at a discrete set of exceptional wave numbers $k$), see for example~\cite{monk+selgas2012}. However, in the TDG method the associated functions are only piecewise smooth and so it is more convenient to use the Neumann-to-Dirichlet (NtD) operator, which has the following expression in terms of modes:
\begin{equation}\label{def:NtD}
 (\mathcal{N} f)|_{\Sigma_{\pm R}}  =   - \sum_{j=0}^{\infty} \frac{i}{\beta_j}  \,\int_{\Sigma_{\pm R}}f\,\overline{\theta_j}\, dS \,  \theta_j     
\quad  \mbox{on } S_R= \Sigma_{-R}\cup  \Sigma_{+ R}  \, ,
\end{equation}	
and defines a bounded operator $ {\cal N} : \widetilde{H}^{-1/2}(S_{ R} )\to H^{1/2}(S_{R} )$, where $ \widetilde{H}^{-1/2} (S_{ R}) $ represents the dual space of $ H^{1/2}(S_{ R})=\{ g|_{S_R}\,;\,\, g\in H^{1/2}(\partial\Omega_{ R})\} $. We then rewrite the radiation condition as
$$
u_{sc}  = {\cal N}   (\partial_{\nbold_{S_{R}}} u_{sc}  )  \quad\mbox{on }S_{R}\, .
$$
Notice that the formulation of the problem written for the scattered field involves a non-zero right hand side in the Helmholtz-like equation inside $\Omega_i$, which  cannot be handled by the classical TDG method. Thus, we write the problem in terms of the total field $u:\Omega_R\to\mathbb{C}$ as follows:
\begin{equation}\label{forwardp}
	\left\{
	\begin{array}{ll}
	\Delta u +k^{2}\,\mathrm{n}\, u     =0 \hspace{.5cm}& \mbox{in }\, \Omega_{i} , \\[1ex]
	\Delta u +k^{2} \, u =0	\hspace{.5cm}& \mbox{in }\, \Omega_{e,R},\\[1ex]
        \jump{ u }_{\nbold_{\partial\Omega_i}} =\boldsymbol{0} \, ,\,\, \jump{ \nabla u }_{\nbold_{\partial\Omega_i}} =0 &  \mbox{on }\, \partial\Omega_i, \\[1ex]
	\displaystyle \partial_{\nbold_{\Gamma_{R}} } u  =0 &\mbox{on } \Gamma_{R},\\[1ex]
	\displaystyle u  = {\cal N}  (\partial_{\nbold_{S_{R}}} u)  +	(u_{inc} - \mathcal{N} (\partial_{\nbold_{S_{R}}} u_{inc}  ) )\hspace{.25cm}&\mbox{on }  S_{ R}.
	\end{array}
     \right.
\end{equation}
Let us emphasize that (\ref{forwardp}) is equivalent to the problem with the radiation condition written using the DtN map  $\mathcal{D}=\mathcal{N}^{-1}$, which in turn is well posed except for, at most, a discrete set of exceptional wave numbers with no finite accumulation points, see the comment above (\ref{def:NtD}). Besides the constraint that $k$ is not a cutoff wave number, we also assume that
$k$ is not one of these exceptional wave numbers.  Note that  the problem is well posed for every wave number when the imaginary part of the contrast $\mathrm{n}$ is positive in a non-empty domain.

 
\section{A modified TDG formulation}\label{sec:varform}

Let $\{\mathcal{T}_h\}_{h>0}$ be a family  of finite element meshes  of $\Omega_R$ that are conforming with the piecewise constant contrast $\mathrm{n}$ and where  $h$ denotes the meshwidth. In particular, for each element $K\in\mathcal{T}_h$ we denote its diameter by   
$h_K={\rm{}diam}(K)$ so that $h=\max_{K\in\mathcal{T}_h} h_K$. We also recall the  chunkiness parameter $\displaystyle s_K=\rho_K/h_K>0$ where $\rho_K$ is the diameter 
of the largest inscribed circle ($m-2$) or sphere ($m=3$) in $K$. We then assume that each mesh $\mathcal{T}_h$ is  non-degenerate and  quasi-uniform:
$$
    \inf_{ h>0} \big\{ \min_{K\in\mathcal{T}_h} s_K \big\} >0 \quad\mbox{and}\quad 
   \sup_{ h>0} \big\{ \displaystyle \max_{K\in\mathcal{T}_h}\frac{h}{ h_K} \big\}
  <\infty\, .
$$
In order to use results derived from Vekua's theory for the Helmholtz operator in the elements of these meshes in~\cite{hiptmair+al2011VekuaTh}, we consider {Lipschitz convex elements}, e.g. triangles or rectangles in dimension $m=2$, and tetrahedra, pyramids or hexahedra in dimension $m=3$. We will refer to the edges (if $m=2$) or faces (when $m=3$) of each element $K$  by {facets}, and to a common one $K_1\cap K_2 $ by {interfacet}. Then, the facets define the skeleton of the mesh and the interfacets its inner part:
$$
\mathcal{E}_h = \bigcup_{K\in\mathcal{T}_h} \{ E\, ; \,\, E\mbox{ a face of  }  K \} \quad\mbox{and}\quad 
\mathcal{E}_h^I =  \!\! \displaystyle\{E\in \mathcal{E}_h\, ; \,\,  E\not\subset\partial\Omega_R \} \, .
$$

We next derive a {Discontinuous Galerkin (DG)} formulation for {Trefftz-type} elements as introduced in \cite{Gittelson+al2009}. As usual for the derivation of a TDG, we start by multiplying  the Helmholtz equation of (\ref{forwardp}) on an element $K\in\mathcal{T}_h$ by the complex conjugate of a smooth test function $v$ and integrating by parts  {twice} to obtain:
\begin{equation}\label{eqn:equal}
\int_K u \, ( \overline{\Delta v+k^2\,\overline{\mathrm{n}}\,v}) \, d\xbold = \int_{\partial K} (-\nbold_K\cdot\nabla u \, \overline{v} + u\, \nbold_K \cdot\nabla\overline{v} ) \, dS_{\xbold } \, ,
\end{equation}
where $\nbold_K$ is the unit normal vector on $\partial K$ that points outside of $K$. 

In the standard TDG approach it is usual to assume that local trial functions satisfy $\Delta u+k^2\,\mathrm{n}\,u=0$ and test functions would need to satisfy $\Delta v+k^2\,\overline{\mathrm{n}}\,v=0$.  Unless $\overline{\mathrm{n}}=\mathrm{n}$, the trial and test spaces differ and this Petrov-Galerkin approach has not been analyzed.  To cope with complex valued $\mathrm{n}$ as is needed to model lossy materials, we instead choose a standard trial and test space and modify the TDG sesquilinear form. In particular, using a local Trefftz trial and test space given by
$$
\begin{array}{c}
V(K)=\{ w\in H^2(K) \, ; \, \, \Delta  {w}+k^2\mathrm{n}\, {w} = 0 \,\,\mbox{in } K\} \, ,
\end{array}
$$ 
the unknown $u|_K\in V(K)$ needs to satisfy 
\begin{equation} \label{eq:varform0_K}
\int_K k^2\,(\mathrm{n}-\overline{\mathrm{n}})\,u\,\overline{v}\, d\xbold = \int_{\partial K} (-\nbold_K\cdot\nabla u \, \overline{v} + u\, \nbold_K \cdot\nabla\overline{v} ) \, dS_{\xbold } \quad\forall v\in V(K)
\, ;
\end{equation}
where we have used the fact that $\Delta \overline{v}=
-\overline{k^2\, {\mathrm{n}}\, v}=-k^2\,\overline{\mathrm{n}}\,\overline{v}$ for each $v\in V(K)$ in (\ref{eqn:equal}). We next consider the global Trefftz space  
$$
\begin{array}{c}
V(\mathcal{T}_h)=\{ w\in L^2(\Omega_R) \, ; \,\, w|_K\in V(K)\,\,\forall K\in\mathcal{T}_h\}\, ,
\end{array}
$$
and add (\ref{eq:varform0_K}) over the elements of the mesh. Then the unknown $u\in V(\mathcal{T}_h)$ must satisfy that
\begin{equation}\label{eq:varform0}
\int_\Omega k^2\,(\overline{\mathrm{n}}-\mathrm{n})\,u\,\overline{v}\, d\xbold +\sum_{K\in\mathcal{T}_h}\!\int_{\partial K}\!  (-\nbold_K\cdot\nabla_h u \, \overline{v} + u\, \nbold_K \cdot\nabla_h\overline{v} ) \, dS_{\xbold } = 0 \quad\forall v\in V(\mathcal{T}_h) \, .
\end{equation}
Hereafter, $\nabla_h$ represents the piecewise gradient in the {broken Sobolev space}
$$
H^1(\mathcal{T}_h)= \{ w\in L^2(\Omega_R)\, ;\,\, w|_K\in H^1(K) \,\,\forall  K\in\mathcal{T}_h\} \, , 
$$
that is, for $w\in H^1(\mathcal{T}_h)$ we have $\nabla_h w\in L^2(\Omega_R) $ defined piecewise by $ (\nabla_h w) |_K =\nabla (w|_K) $  in each $ K\in\mathcal{T}_h $. We can  rewrite (\ref{eq:varform0}) using the  \emph{DG magic formula}, cf.~\cite{MelenkEsterhazy12}. Using also the fact that the exact solution of (\ref{forwardp}) satisfies that $\jumpn{u} =   \boldsymbol{0}  $ and $ \jumpn{\nabla_h u} = 0 $  for all $ E\in\mathcal{E}_h^I $ we see that $u\in V(\mathcal{T}_h)$ satisfies
\begin{equation}\label{eq:varform1pre}
\begin{array}{l}
\displaystyle\int_\Omega k^2\,(\overline{\mathrm{n}}-\mathrm{n})\,u\,\overline{v}\, d\xbold+\sum_{E\in\mathcal{E}_h^I }  \int_{E} (  {u} \,\jumpn{\nabla_h\overline{v}  } -  \nabla_h u\cdot \jumpn{\overline{v}  } )  \, dS_{\xbold }  \\
\displaystyle\quad +\sum_{E\in\mathcal{E}_h\setminus\mathcal{E}_h^I } \int_{E} (   {u}  \, {\nabla_h\overline{v}  }\cdot\nbold -   {\nabla _h u  } \cdot\nbold \,  {\overline{v}  } )  \, dS_{\xbold } \, = \, 0\quad \forall v\in V(\mathcal{T}_h)\, .
\end{array}
\end{equation}
Equation (\ref{eq:varform1pre}) involves: $u$ in the absorbing parts of the scatterer; and only traces of the original unknown $u$ and of the auxiliary field $\nabla_h u$ on facets that are not inside the absorbing parts of  the scatterer. To handle such traces at the discrete level where the trial and test functions are discontinuous, we start by following \cite{Gittelson+al2009} and replace them  formally by suitable numerical fluxes, in which we penalize undesired jumps.  
In particular, we define the following {primal fluxes} on the  facets:
\begin{itemize}
\item if $E\in\mathcal{E}_h ^I$, we enforce the continuity of traces and normal derivatives by taking
$$
 \begin{array}{l}
\displaystyle \hat{u} = \av{ u}  -i\, \frac{b}{k}\, \jumpn{ \nabla _h u  } \, ;
\\[1ex]
\displaystyle ik\, \hat{\boldsymbol{\sigma}} =    \av{ \nabla_h u } +iak\, \jumpn{u }  \, ;
\end{array}
$$
\item if $E\in\mathcal{E}_h$ is on  $\Gamma_R$, the  sound hard condition $\partial_{\nbold}u=0$ suggests us to consider
$$
\begin{array}{l}
\displaystyle \hat{u} =  u-i\, \frac{d_1}{k}\,  { \nabla _h u  }\cdot\nbold \, ;
\\
\displaystyle ik\, \hat{\boldsymbol{\sigma}} = \boldsymbol{0} \, ;
\end{array}
$$
\item when $E\in\mathcal{E}_h$ is on $S_{R}$, we take into account the decaying condition $ \displaystyle  u  = {\cal N}  (\partial_{\nbold } u)  +	u_{inc} - \mathcal{N} (\partial_{\nbold } u_{inc}   ) $ by defining
$$
\begin{array}{l}
\displaystyle \hat{u} =  {\cal N}   (\nabla_h u \cdot\nbold)  + u_{inc} - \mathcal{N} (\partial_{\nbold } u_{inc}  )\\
\qquad\qquad -  i\, k\, d_2 \,\mathcal{N}^* \left( {\cal N}   (\nabla_h (u-u_{inc}) \cdot\nbold)  - (u-u_{inc}) \right) ;
\\[1ex]
\displaystyle ik\, \hat{\boldsymbol{\sigma}} =    { \nabla_h u } -i\, k\, d_2   \left( {\cal N}   (\nabla_h (u-u_{inc}) \cdot\nbold)  -(u-u_{inc})  \right)  \nbold\, ,
\end{array}
$$
where $\mathcal{N}^* : L^2(S_{R})\to L^2(S_{R})$ represents the adjoint of the NtD map:
\begin{equation*}
\begin{array}{l}
\mathcal{N}^*  f  =   \displaystyle \sum_{j=0}^{\infty}    \frac{i}{\,\overline{\beta_j}\,}   \, \int_{\Sigma_{\pm R}} f\,\overline{\theta_j}\, dS \, {\theta_j}  \quad\mbox{on }\Sigma_{\pm R}\, .
\end{array} 
\end{equation*}	
This flux is novel and chosen so that it provides control over the jump in the normal derivative across the truncation boundaries.
\end{itemize}
In general, the flux parameters $a,b ,d_1,d_2 \in\R$ should be  strictly positive piecewise constant functions on $\mathcal{E}_h$ which are bounded below uniformly with respect to $h$. However, for simplicity we will choose them to be  strictly positive constants independent of $K\in\mathcal{T}_h$ for the analysis, e.g. equal to $1/2$ as usual in UWVF formulations. We will later state further theoretical requirements on them and also investigate their role numerically, see Sections~\ref{sec:convTrefftz} and \ref{sec:numerics}. 

Substituting the proposed fluxes in (\ref{eq:varform1pre}) and collecting terms
we see that $u\in V({\cal T}_h)$ satisfies
\begin{equation}\label{eq:varfor-uexact}
\begin{array}{l}
\mathcal{A}_h (u ,v  ) \, =\,  
 \displaystyle  {L_h(v)} \quad \forall v\in V(\mathcal{T}_h) \, .
\end{array}
\end{equation}
Here the sesquilinear form is defined for $w,v \in V(\mathcal{T}_h) $ by
\begin{equation}\label{Ahdef}
\begin{array}{l}
\mathcal{A}_h (w ,v  ) \, =\,  
\displaystyle 2i \int_\Omega k^2\, \Im( \mathrm{n})\,w\, \overline{v}\, d\xbold-\int_{\Gamma _R}  \!\!  (w-i\frac{d_1}{k}\nabla_h w\cdot\nbold )\, {\nabla_h\overline{v }  }\cdot\nbold  \, dS_{\xbold }\\[1ex]
\displaystyle\quad +\sum_{E\in\mathcal{E}_h^I}  \!\int_{E} \!\! \big(  (-\!\av{ w}\! +i\frac{b}{k} \jumpn{\nabla_h w} )  \jumpn{\nabla_h\overline{v }  }  \!  + ( aik\jumpn{w}\! +\!\av{\nabla_h w} )  \jumpn{\overline{v }  }\! \big)  \, dS_{\xbold }   \\[1ex]
\displaystyle\quad
- \int_{S_{R }}\!\!\big( \mathcal{N}( \nabla_hw\cdot\nbold) \,\nabla_h \overline{v}\cdot\,\nbold- \nabla_h w\cdot\nbold\, \overline{v}  \, \big)\, dS_{\xbold}\\[1ex]
\displaystyle\quad + \int_{S_{R }}\!\! d_2\, i\, k\,   \big(\mathcal{N}  (\nabla_hw\cdot\nbold)-w \big) \, \overline{\big(\mathcal{N}(\nabla_h {v}\cdot\nbold )-{v}\big)}   \, dS_{\xbold} \, ,
\end{array}
\end{equation}
and the antilinear map $L_h$ is defined for all $v\in V({\cal T}_h)$
by
$$
\begin{array}{l}
 \displaystyle  L_h ( v  ) = \int_{S_{R }}\!\! \Big(   {( u_{inc} -\mathcal{N}(\partial_{\nbold} u_{inc}) )} \, \overline{\nabla_h  v \cdot\,\nbold} \\
 \hspace*{2cm} \displaystyle  +   \,  d_2\, i\, k \,  \big(\mathcal{N} ( \partial_{\nbold} u_{inc}  )-u_{inc} \big) \, \overline{ \big(\mathcal{N} (\nabla_h {v}\cdot\nbold)- {v} \big) } \, \Big)\, dS_{\xbold}   \, .
\end{array}
$$

Accordingly, when the Trefftz space $V(\mathcal{T}_h)$ is discretized with a finite dimensional subspace $V_h(\mathcal{T}_h)$, we propose the following discrete modified TDG formulation:
\begin{equation}\label{forward_Trefftz-DG}
\left\{
\begin{array}{l}
\mbox{Find}\,\, u_h\in V_h (\mathcal{T}_h) \,\, \mbox{s.t. }   \\[1ex]
\mathcal{A}_h (u_h,v_h) =  {L_h ( v_h )} \quad \forall v_h\in   V_h (\mathcal{T}_h)
\end{array}
\right.
\end{equation}
For later use, we note that, by construction, the proposed TDG formulation is consistent.
\begin{lemma}
\label{lem:varfor-uexact}
The solution $u\in H^2(\Omega_{R})$ of the original problem (\ref{forwardp}) belongs to $V(\mathcal{T}_h)$ and satisfies (\ref{eq:varfor-uexact}).
\end{lemma}

\begin{remark}
Although the emphasis of our paper is on penetrable scatterers, the TDG method can handle impenetrable scatterers as well. For example, for a sound soft scatterer the acoustic field $u$ satisfies
$$
\begin{array}{l}
\mbox{Find }\, u:\Omega_R\to\mathbb{C}\,\, \mbox{such that} \\
	\left\{
	\begin{array}{ll}
	\Delta u +k^{2}  u =0	\hspace{.5cm}& \mbox{in }\, \Omega_{e,R}\\
 	\displaystyle \partial_{\nbold_{\Gamma_{R}} } u  =0 &\mbox{on } \Gamma_{R}\\
	\displaystyle u  = {\cal N}  (\partial_{\nbold_{\Sigma_{ R}}} u)  +	(u_{inc} - \mathcal{N} (\partial_{\nbold_{\Sigma_{R}}} u_{inc}  ) )\hspace{.25cm}&\mbox{on }  S_{R}
	\end{array}
     \right.
\end{array}
$$
As usual we consider the associated skeleton $\mathcal{E}_{e,h}$ and its interior part $\mathcal{E}_{e,h}^I$, as well as the global Trefftz space on $\Omega_e$ given by 
$$
V_e(\mathcal{T}_h)=\{ v\in L^2(\Omega_{e,R}) \, ; \,\, v|_K\in V(K)\,\,\forall K\in\mathcal{T}_h\}\, .
$$ 
We can reason as for the penetrable target, and take the same fluxes as before for facets in $\mathcal{E}_{e,h}^I$ and on $\Gamma_R\cup S_{R}$, whereas for the facets on $\partial \Omega_i$ we may use~\cite{Gittelson+al2009}
$\hat{u} = 0$ and $
ik\, \hat{\boldsymbol{\sigma}} =    \nabla_h u  +iak\,  u\boldsymbol{n}   \, .
$
We then get an analogous TDG formulation in which the sesquilinear form is now
\begin{eqnarray*}
&&\mathcal{A}_{e,h} (w ,v  ) =\\&&
\qquad\sum_{E\in\mathcal{E}_{e,h}^I} \!\!\! \int_{E}\!\! \big(  (-\av{ w}+i\frac{b}{k} \,\jumpn{\nabla_h w} ) \,\jumpn{\nabla_h\overline{v }  }  + ( aik\jumpn{w}+\av{\nabla_h w} )  \, \jumpn{\overline{v }  } \big)  \, dS_{\xbold }   \\
&&\qquad + \int_{\partial\Omega_i}\!\! \big(  \nabla_h u\cdot\nbold +aik u \big)  \,\overline{v}  \, dS_{\xbold} -\int_{\Gamma _R}    (w-i\frac{d_1}{k}\nabla_h w\cdot\nbold )\, {\nabla_h\overline{v }  }\cdot\nbold  \, dS_{\xbold }\\
&&\qquad  
- \int_{S_{ R }}\!\!\big( \mathcal{N}( \nabla_hw\cdot\nbold) \,\nabla_h \overline{v}\cdot\,\nbold- \nabla_h w\cdot\nbold\, \overline{v}  \\&&\qquad\qquad- d_2\, i\, k\,   (\mathcal{N}  (\nabla_hw\cdot\nbold)-w ) \, \overline{(\mathcal{N}(\nabla_h {v}\cdot\nbold )-{v})}  \, \big)\, dS_{\xbold}
\end{eqnarray*}
and the antilinear map is as before. The upcoming analysis covers this case as well with obvious modifications, so extending standard convergence theory to NtD terminated scattering for an impenetrable scatterer.
\end{remark}


\section{Study of the discrete TDG variational formulation}\label{sec:convTrefftz}

In order to analyze the discrete TDG formulation (\ref{forward_Trefftz-DG}), we first show that the sesquilinear form $ \mathcal{A}_h (w,v ) $ is bounded and coercive for suitable norms in $V(\mathcal{T}_h)$. Then, we make use of a duality argument to study the convergence of the discrete solutions in the $L^2$ norm~\cite{MelenkEsterhazy12,hiptmairetal2011p2d}. Finally, we provide error bounds for a discretization based on plane waves under a regularity assumption on the forward problem.  The argument follows \cite{monk+wang1999,hiptmairetal2011p2d} closely except for special steps to deal with the NtD transmission conditions, and the new loss term.


\subsection{Analysis of the discrete TDG formulation}\label{ssec:studyAh}

We start by showing that the assumption on the strict positivity of the parameters $a,b,d_1,d_2\in\mathbb{R}$ is enough to guarantee that the sesquilinear form is coercive. To this end, let us consider $v\in V(\mathcal{T}_h)$, then (\ref{Ahdef}), simplified using the DG-magic formula, gives
$$
\begin{array}{l}
\Im ( \mathcal{A}_h (v ,v ) ) \, =\, 
\displaystyle 2\int_\Omega k^2\, \Im (\mathrm{n})\, |v|^2 \, d\xbold+
\displaystyle\sum_{E\in\mathcal{E}_h^I}  \int_{E} \!
\big(  a \, k\,| |\!\!\left[ v  \right]\!\! | _{\nbold} |^2   
+ \frac{b}{k}\, | |\!\!\left[ \nabla _h v  \right]\!\! |_{\nbold}  | ^2
  \big) \, dS_{\xbold } 
\\
\quad +\displaystyle  
\frac{d_1}{k}  \,   \int_{\Gamma_R}  \!|\nabla _h v\cdot\nbold |^2    \, dS_{\xbold } 
\\
\displaystyle \quad + \int_{S_{ R }} \!\!\big(  d_2 \, k \, |\mathcal{N} ( \nabla_h v\cdot\nbold )-v|^2   -\Im ( {\mathcal{N}} ( \nabla_h v\cdot\nbold) \,\nabla_h \overline{v}\cdot\nbold  ) \big)\, dS_{\xbold}  
 \, .
  \end{array}
$$
To understand the behavior of the terms on the auxiliary cross-sections, we rewrite them with modes:
$$
\begin{array}{l}
\Im ( \mathcal{A}_h (v ,v ) ) \, =\,  
\displaystyle 2\int_\Omega k^2\, \Im (\mathrm{n})\, |v|^2 \, d\xbold+
\displaystyle\sum_{E\in\mathcal{E}_h^I}  \int_{E} \!
\big(   a  k\,| |\!\!\left[ v  \right]\!\! | _{\nbold} |^2   
+\frac{b}{k}\, | |\!\!\left[ \nabla _h v  \right]\!\! |_{\nbold}  | ^2
  \big) \, dS_{\xbold } 
\\
\displaystyle \qquad  +  \frac{d_1}{k}\, \int_{\Gamma_R }\! |\nabla_h v\cdot\nbold |^2  \, dS_{\xbold } 
\\
\displaystyle \qquad + \sum_{j=0}^{N_{\rm{}pr}} \frac{1}{|\beta_j|}  \, |( \nabla_h v \cdot\nbold)_j ^- |^2  +d_2\, k \, \sum_{j=0}^{\infty}  | \frac{-i}{\beta_j} \, ( \nabla_h v \cdot\nbold )_j^- -v_j^- |^2 
\\
\displaystyle \qquad + \sum_{j=0}^{N_{\rm{}pr}} \frac{1}{|\beta_j|}  \, |( \nabla_h v \cdot\nbold)_j ^+ |^2  +d_2\, k \, \sum_{j=0}^{\infty}  | \frac{-i}{\beta_j} \, ( \nabla_h v \cdot\nbold )_j^+ -v_j^+ |^2 
 \, ,
\end{array}
$$
where we denote 
\begin{equation}\label{notation:modes}
v_j^{\pm}= \int_{\Sigma_{\pm R}} \!\! v  \, \overline{\theta_j}\, dS
\quad\mbox{and}\quad
( \nabla_h v \cdot\nbold )_j^{\pm}= \int_{\Sigma_{\pm R}} \!\!\! \nabla_h v \cdot\nbold \, \overline{\theta_j}\, dS \quad \mbox{for } j=0,1,\ldots
\end{equation}
Hence $\Im ( \mathcal{A}_h (v ,v ) ) $ is non-negative for every $v\in V(\mathcal{T}_h)$, and allows the introduction of the mesh dependent seminorm $ 
||| v|||_{h}  = (   \Im ( \mathcal{A}_h (v ,v ) ) ) ^{1/2}$. 
\begin{lemma}\label{lem:seminormisnorm}
The seminorm $|||\cdot|||_{h}$ defines a norm in $V(\mathcal{T}_h)$.
\end{lemma}
\begin{proof}
Let us consider $v\in V(\mathcal{T}_h)$ such that $||| v|||_{h} =0$. Then
$$
\begin{array}{ll}
v=0&\mbox{in } \mbox{supp}(\Im(\mathrm{n}))\, ,\\[1ex]
\displaystyle
 |\!\!\left[ v  \right]\!\! | _{\nbold} = |\!\!\left[ \nabla _h v  \right]\!\! |_{\nbold} =0   \quad &\mbox{across every } E\in\mathcal{E}_h^I\, ,  \\[1ex]
 \displaystyle
 \nabla _h v\cdot\nbold = 0 \quad & \mbox{on } \Gamma_R\, ,\\[1ex]
 \displaystyle
 \mathcal{N} (\nabla _h v\cdot\nbold )= v \quad & \mbox{on } S_{R}\, .
 \end{array}
$$
In particular $v\in H^1(\Omega_R)$ solves  problem  (\ref{forwardp}) for $u_{inc}=0$ and, since we are assuming that this problem  is well posed for the wave number under study, it follows that necessarily $v=0$ in $\Omega_R$. 
\end{proof}

As a straightforward consequence of Lemma~\ref{lem:seminormisnorm}, there is at most one solution of the TDG formulation (\ref{forward_Trefftz-DG}) and hence,
given the existence of a standard solution to the waveguide problem, this formulation is well posed.

To study the continuity of the sesquilinear form defined in (\ref{Ahdef}), we rewrite the first of the terms on the truncation walls using modes with the notation (\ref{notation:modes}):
$$
\begin{array}{l}
\displaystyle\int_{\Sigma_{\pm R }}\big( \mathcal{N}( \nabla_hw\cdot\nbold) \,\nabla_h \overline{v}\cdot\,\nbold -  \nabla_h w\cdot\nbold\, \overline{{\cal N}(\nabla_h {v}\cdot\,\nbold)}\big) \, dS_{\xbold}\\    
\displaystyle \quad =\sum_{j=0}^\infty\left(\frac{-i}{\beta_j} - \frac{i}{\overline{\beta_j}}\right) (\nabla_h w\cdot \nbold)_j^{\pm}\, \overline{(\nabla_h v\cdot \nbold)_j^{\pm}} 
=\sum_{j=0}^{N_{\rm{}pr}} \frac{-2i}{\beta_j}\, (\nabla_h w\cdot \nbold)_j^{\pm}\, \overline{(\nabla_h v\cdot \nbold)_j^{\pm}} \, .
\end{array}
$$
Then we can bound the sesquilinear form by the triangle inequality and the Cauchy-Schwartz inequality, using the  augmented norm in $V(\mathcal{T}_h)$ defined by
$$
\begin{array}{l}
\displaystyle ||| v|||_{h,+}^2 \, =\,  ||| v|||_{h}^2   +  \sum_{E\in\mathcal{E}_h^I}  \int_{E} \! \big(  ( \frac{k}{b} \, | \av{ v}| ^2     +  \frac{1}{ak} \, |\av{\nabla_h v}|^2  \big)  \, dS_{\xbold }   \\
\quad\displaystyle + \int_{\Gamma _R}    \!\frac{k}{d_1}\, |v|^2   \, dS_{\xbold } 

+\sum_{j=0}^{\infty} \frac{1}{d_2k} \,  
(|(\nabla_h v\cdot\nbold )_j^-  |^2 + |(\nabla_h v\cdot\nbold )_j^+ |^2)  
 \, .
  \end{array}
$$
We summarize this bound of the sesquilinear form in the following Lemma.
\begin{lemma}\label{lem:propAh}
It holds that, for all $v,w\in V_h(\mathcal{T}_h)$,
$$ 
|\mathcal{A}_h(w,v)| \leq \frac{1}{2}\,  |||w|||_{h,+}\,  |||v|||_{h} \, .
$$
\end{lemma}


\subsection{Convergence of the discrete TDG solutions}\label{ssec:convTrefftz}
We next study the convergence of the approximate solution $u_h$ toward the exact solution $u$. Indeed, we bound the approximation error first in the mesh dependent norm $|||\cdot|||_{h}$, and then in the more usual norm  $||\cdot ||_{0,\Omega_R}$. 


We start by noticing that, thanks to the coercivity and boundness of the sesquilinear form, as well as the consistency of the formulation:
$$
\begin{array}{l}
\displaystyle
||| u_h-v_h ||| ^2 _{h} = \Im ( \mathcal{A}_h (u_h-v_h , u_h-v_h  )) 
= \Im ( \mathcal{A}_h (u -v_h , u_h-v_h  )) \\[1ex]
\displaystyle\qquad\qquad \leq \frac{1}{2} \, ||| u-v_h |||_{h,+} \,||| u_h-v_h||| _{h} \, ,
\end{array}
$$
which implies that
$$
\begin{array}{l}
\displaystyle
||| u_h-v_h ||| _{h} \leq \frac{1}{2} \, ||| u-v_h |||_{h,+} \, ,
\end{array}
$$
for every $v_h\in V_h(\mathcal{T}_h)$.  By applying the triangular inequality, we deduce the following result, see \cite{MelenkEsterhazy12}.
\begin{lemma}\label{lem:quasiopt}
The discrete solution $u_h$ satisfies the error bound
$$
\begin{array}{l}
\displaystyle
||| u-u_h |||  _{h} \,\leq\, \frac{3}{2} \, \inf_{v_h\in V_h(\mathcal{T}_h) }  ||| u-v_h |||_{h,+}  \, .
\end{array}
$$
\end{lemma}
\begin{remark}
So far, the assumptions on the family of finite element meshes stated at the beginning of Section~\ref{sec:varform} have not been used. That is, the mesh dependent coercivity and continuity properties of the sesquilinear form $\mathcal{A}_h(\cdot,\cdot)$ stated in Lemma~\ref{lem:propAh}, as well as the mesh dependent 
error bound of  Lemma~\ref{lem:quasiopt}, are still fulfilled for general partitions of $\Omega_R$ that are conforming with respect to the piecewise constant $\mathrm{n}$.
\end{remark}


The error bound stated in  Lemma~\ref{lem:quasiopt} is written in terms of mesh dependent norms. We next use a duality argument to
obtain an estimate in the mesh independent $L^2(\Omega_R)$ norm (see \cite{monk+wang1999,hiptmairetal2011p2d}). More precisely, let $u\in H^1 (\Omega_R)$ be the solution of (\ref{forwardp}) and $w_h\in V_h(\mathcal{T}_h)$ any discrete Trefftz function; we  seek a suitable upper bound of $||u-w_h||_{L^2(\Omega_R)}$ in terms of $|||w_h|||_{h}$. To this end,  for each $\varphi\in L^2(\Omega_R)$ we consider the solution  $U_{\varphi}:\Omega_R\to\mathbb{C}$ of the following dual problem  of (\ref{forwardp}):
\begin{equation}\label{forwardp_adj}
\left\{	\begin{array}{ll}
 \displaystyle \Delta U_{\varphi}+k^{2}\,\overline{\mathrm{n}}\, U_{\varphi}=  \varphi|_{\Omega_i}  \hspace{.5cm}&\mbox{in } \Omega_i, \\[1ex]
		 \Delta U_{\varphi}+k^{2} U_{\varphi}=  \varphi|_{\Omega_e}    \hspace{.5cm}
		&\mbox{in } \Omega_{e,R}, \\[1ex] 
			\displaystyle \jump{ U_{\varphi}} =  0  \, ,\quad 
			\displaystyle \jump{  \partial_{\nbold }U _{\varphi} } = 0  \hspace{.5cm}&\mbox{on }\partial \Omega_i ,\\[1ex]
		\displaystyle	 \partial_{\nbold } U _{\varphi} =0 &\mbox{on } \Gamma_{R},\\[1ex]
	 \displaystyle U_{\varphi}  = {\cal N} ^*  (\partial_{\nbold } U_{\varphi} )  \hspace{.5cm}&\mbox{on } S_{R}.
		\end{array}\right.
\end{equation}
This adjoint problem is  well posed for $\varphi\in L^2(\Omega_R)$ under our assumptions on the wave number $k$, the geometry $\Omega_R$ and the contrast $\mathrm{n}$. Moreover, since $\Delta U_\phi\in L^2(\Omega)$ with nonhomogeneous boundary data, we  have $U_{\varphi}\in H^2(\Omega_R)$ and we assume that we have the following stability bounds:
\begin{equation}\label{boundstab_Uphi}
    \begin{array}{l}
    |U_{\varphi}|_{1,\Omega_R} + k\, ||U_{\varphi}||_{0,\Omega_R} \leq C_1\, d_{\Omega_R }\, ||\varphi||_{0,\Omega_R}\, ,
    \\[1ex]
    |U_{\varphi}|_{2,\Omega_R} \leq C_2\, (1+k\,  d_{\Omega_R })\, ||\varphi||_{0,\Omega_R}\, ,
    \end{array}
\end{equation}
where $C_1,C_2>0$ depend on the shape of $\Omega_R$ and $ d_{\Omega_R }$ is the diameter of $\Omega_R$ (that is, $d_{\Omega_R}=2R$ for $R>0$ big enough), and are independent of $k$. We refer to the series of papers \cite{DemEtAl_2023,DemEtAl_2024,Norbert_2024} in which the stability analysis of this problem is currently being developed.

As for the derivation of the TDG formulation   (\ref{forward_Trefftz-DG}), for each $w\in V (\mathcal{T}_h)$ we integrate by parts twice on each element of the mesh resulting in the following expression
\begin{equation}\label{eq:wadj_withu}
\begin{array}{rl}
\displaystyle \int _{\Omega_R } \!\!\! (u-w)\, \overline{\varphi}\, d\xbold  
&= \displaystyle\sum_{K\in\mathcal{T}_h}  \int _{ K  } (u-w)\, (\Delta \overline{U_{\varphi}}+k^{2}\,\mathrm{n}\, \overline{U_{\varphi}} ) \, d\xbold  
\\
 &\displaystyle =\sum_{E\in\mathcal{E}^I_h  } \! \int_{E} \!\!  \big( \overline{U_{\varphi}}   \, \jumpn{\nabla_h  (w-u) }  \!  -  \!   \jumpn{ w-u} \! \cdot\! \overline{\nabla  U_{\varphi}}\big)  \,   dS_{\xbold} \\
 & \quad \displaystyle   +    \int_{\Gamma_R} \!\!\!\nabla_h (w-u) \cdot\nbold \, \overline{U_{\varphi}} \,   dS_{\xbold} \\ 
 &\quad
 \displaystyle  +  \int_{S_{ R}}\!\!\! \big( \overline{\nabla U_{\varphi}\cdot\nbold} \, ( u-w ) -\nabla_h  (u-w) \cdot\nbold \, \overline{U_{\varphi}} \big) \,   dS_{\xbold}  \, ,
\end{array}
\end{equation}
where we have simplified some terms thanks to the fact that $u\in H^1(\Omega_i\cup\Omega_{e,R})$ is the solution of the forward problem (\ref{forwardp})  and  $U_{\varphi}\in H^1(\Omega_{R})$ the solution of the auxiliary dual problem (\ref{forwardp_adj}). We can also use this fact to further understand the behavior of the terms on the truncation walls, indeed we can rewrite that $u-u_{inc}={\cal N } (\partial _{\nbold } (u-u_{inc})) $ and $\mathcal{N}^* (\partial_{\nbold}U_{\varphi}) =  U_{\varphi}$ on $S_{R}$ in terms of modes with the notation (\ref{notation:modes}):
\begin{equation*}
\begin{array}{l}
\displaystyle (\nabla_h (u-u_{inc})\cdot\nbold)_j^{\pm} =   i\, \beta_j \, (u-u_{inc})_j^{\pm}   
\quad\mbox{and}\quad (\nabla U_{\varphi}\cdot\nbold)_j^{\pm } =  -i   \overline{\beta_j}\,   (U_{\varphi})_j^{\pm }
\end{array} 
\end{equation*}
for all $j=0,1,\ldots$ We then have that
$$
\begin{array}{l}
\displaystyle
 \int_{S_{R}} (\overline{\nabla U_{\varphi} \cdot\nbold }\, (u-w) -\nabla_h (u-w) \cdot\nbold \,\overline{U_{\varphi}}) \,   dS_{\xbold} \\
 \qquad\displaystyle =   
 \int_{S_{R}}  (\mathcal{N} (\nabla_h (w-u )\cdot\nbold)   - (w-u) )  \,  \overline{  \nabla  U_{\varphi } \cdot\nbold } \, dS_{\xbold} \, .
  \end{array} 
$$
Substituting this in (\ref{eq:wadj_withu}) we get that
$$
\begin{array}{l}
\displaystyle \int _{\Omega_R } \!\! (u-w)\, \overline{\varphi} \, d\xbold  =\sum_{E\in\mathcal{E}^I_h  }  \!\int_{E}  \!\!\! \big(\overline{U_{\varphi}}   \, \jumpn{\nabla_h (w-u) }  -   \jumpn{ w-u}  \!\cdot \!\nabla \overline{U_{\varphi}} \big) \,   dS_{\xbold} \\
\qquad \qquad\displaystyle   +    \int_{\Gamma_R} \!\!\nabla_h (w-u) \cdot\nbold \, \overline{U_{\varphi}} \,   dS_{\xbold} \\
\qquad \qquad\displaystyle   + \,  \int_{S_{R}} \!\!\! \big(\mathcal{N} (\nabla_h (w-u)\cdot\nbold)   -(w-u)\big)  \,   \overline{\nabla  U_{\varphi }\cdot\nbold } \, dS_{\xbold}  \,  .
\end{array}
$$
This can be bounded by applying the triangle and Cauchy-Schwartz inequalities:
$$
\begin{array}{l}
\displaystyle |\int _{\Omega_R } (u-w)\, \overline{\varphi} \, d\xbold | \, \leq  \,   \frac{1}{2} \, \mathcal{B} (U_{\varphi}  )^{1/2}  \, ||| w-u |||_{h}
\,  ,
\end{array}
$$
where 
$$
\begin{array}{l}
\displaystyle\mathcal{B} (U_{\varphi}  ) =
\sum_{E\in\mathcal{E}^I_h   } \int_{E}\!\!\big(   \frac{k}{b} \, | U_{\varphi}  |^2  
+ \frac{1}{ak} \,  | \partial_{\nbold} U_{\varphi}|^2  \big) \,   dS_{\xbold}  
\\
\displaystyle \qquad\qquad \qquad  +    \int_{\Gamma_R} \!\!\frac{k}{d_1}\, | U_{\varphi} |^2  \,    dS_{\xbold}   +   \int_{S_{R} }\!\! \frac{1}{d_2k} \, | \nabla U_{\varphi }\cdot\nbold |^2 \, dS_{\xbold} 
\, .
\end{array}
$$
In turn, we can bound $\mathcal{B} (U_{\varphi})$ using the trace inequality
\begin{equation}\label{eq:traceineq}
||v||_{0,\partial K}^2  \leq C_K \, ||v||_{0,K} \, (h_K^{-1}||v||_{0,K}+|v|_{1,K})\quad\forall v\in H^1( K) \, ,
\end{equation}
where $C_K>0$ depends only on bound on the chunkiness parameter for the element $K$, and hence  is bounded uniformly for all the elements $ K\in \cup_{h>0}\mathcal{T}_h$ thanks to the shape-regularity assumption on the family of finite element meshes at the beginning of (\ref{sec:varform}). Indeed, when $U_{\varphi}\in H^2(\Omega_R)$ satisfies the stability bounds (\ref{boundstab_Uphi}), we have that
$$
\begin{array}{l}
\displaystyle\mathcal{B} (U_{\varphi}) \leq  
C \, d_{\Omega_R}^2 \, \Big( 1+ \frac{1}{hk} \Big)  \, || \varphi ||_{0,\Omega_R} ^2 
\, ,
\end{array}
$$
and therefore
$$
\begin{array}{l}
\displaystyle |\int _{\Omega_R } (u-w)\, \overline{\varphi}\, d\xbold |  \leq C \, d_{\Omega_R}  \, \Big( 1+ \frac{1}{hk} \Big)^{1/2}  \, || \varphi ||_{0,\Omega_R}   \, |||w-u |||_{h} \, .
\end{array}
$$
Since this holds for every $\varphi\in L^2 (\Omega_R)$, we have that
$$
\begin{array}{l}
\displaystyle || u-w||_{0,\Omega_R }  \leq  C\, d_{\Omega_R}\,\Big( 1+ \frac{1}{hk} \Big)^{1/2} \,
 |||u-w|||_{h}  \, .
\end{array}
$$
In particular, for  $w=u_h$, we deduce the following result.
\begin{lemma}\label{lem:quasioptL2}
The exact and discrete solutions of (\ref{forwardp}) and (\ref{forward_Trefftz-DG}) 
satisfy that
$$
\begin{array}{l}
\displaystyle
 || u-u_h||_{0,\Omega_R }  \,\leq\,   C\,   \Big( 1+ \frac{1}{hk} \Big)^{1/2}\, |||u-u_h|||_{h} \, ,
\end{array}
$$
where $C>0$ is independent of the wave number $k$, the mesh size $h>0$ and $V_h({\cal T}_h)$.
\end{lemma}

Summing up, Lemmas~\ref{lem:quasiopt} and \ref{lem:quasioptL2} provide the  following  error bound constituting the main result of our paper.

\begin{theorem}\label{lem:quasioptL2ok}
The exact and discrete solutions of (\ref{forwardp}) and (\ref{forward_Trefftz-DG}) satisfy that
$$
\begin{array}{l}
\displaystyle
 || u-u_h||_{0,\Omega_R }  \,  \leq \,  C\,   \Big( 1+ \frac{1}{hk} \Big)^{1/2} \,  \inf_{v_h\in V_h(\mathcal{T}_h) }  ||| u-v_h |||_{h,{+}}  
 \, ,
\end{array}
$$
where $C>0$ is independent of the wave number $k$,  the mesh size $h>0$ and $V_h({\cal T}_h)$.
\end{theorem}


\subsection{Plane wave discretization}\label{ssec:planewaves}
We next study the discretization error  in the augmented norm $||| \cdot |||_{h,+}$ for fields in $ V(\mathcal{T}_h)$. For a discretization based on plane or evanescent waves using $N_p$ different directions $\{\dbold_j\}_{j=1}^{N_p}$, the space of the discrete solutions on the element $K\in\mathcal{T}_h$ is:
\begin{equation}
V_h(K) = {\rm{}span} \{ \exp(i\, k\sqrt{n}\,(\xbold-\xbold_{0,K})\cdot\dbold_j),\;j=1\cdots N_p\}\, ,\label{VhKdef}
\end{equation}
where $\xbold_{0,K}$ is the centroid of $K$.  Despite the possible presence of evanescent factors, we generically refer to this as a plane wave basis.

In the following, we take the dimension of the local discretization space to be
$$
N_p=\begin{cases}
    2p+1\quad & \mbox{if } m=2\, ,\\
    (p+1)^2 \quad & \mbox{if } m=3\, , 
\end{cases} 
$$
for some $p\in\mathbb{N}$, and the directions are:
\begin{itemize}
    \item when $m=2$ the directions $\{\dbold_j=(\cos\alpha_j,\sin\alpha_j)  \}_{j=1}^{N_p} $ have angles of incidence $\alpha_j$ evenly distributed in $[0,2\pi)$;
    \item when $m=3$, the choice of directions is analyzed in \cite{moiola_PhD} and an almost optimal set of directions (points on the unit sphere) can be downloaded from \cite{sloan_wp}.
\end{itemize}
In this paper, we opt to use the same number of directions in all the cells and choose the directions identically on each cell for simplicity of the exposition. We refer to \cite{moiola_PhD} and references therein for a discussion on more general choices of well distributed directions.

\subsection{Error bound for plane wave discretizations}

We next study the order of convergence for $h$ and $p$-refinement, respectively:
\begin{itemize}
\item Convergence under \emph{mesh refinement} $h\to 0^+$  for a fixed number of directions $N_p$ per cell. 
\item Convergence under \emph{increasing order} $p\to\infty$ for a partition $\mathcal{T}_h$ fixed. 
\end{itemize}
To obtain suitable error estimates out of the error bound provided by Lemma~\ref{lem:quasioptL2ok}, we investigate the asymptotic behavior of the discretization error $ \displaystyle\inf_{v_h\in V_h(\mathcal{T}_h) } \!\! ||| u-v_h |||_{h,+}$. Let us first notice that the augmented mesh dependent norm can be bounded using the $L^2$ norm in the scatterer together with $L^2$ norms on the skeleton:
$$
\begin{array}{l}
\displaystyle |||v|||_{h,+}^2 \, \leq \, 2\, || \Im(\mathrm{n})||_{\infty,\Omega_i} \,  ||v||_{0,\Omega_i}^2 +
 \max\Big\{  \frac{2ab\! +1}{b} , \frac{1}{d_1} , d_2 \Big\}  \, k\, ||v||_{0,\mathcal{E}_h} ^2   \\[1ex]
\qquad\displaystyle + \max\Big\{ \frac{2ab+1}{a}  , d_1 , \frac{1}{d_2}+\frac{k}{\beta_{N_p}}+ \, \frac{d_2\, k^2}{\min\{\beta_{N_p},|\beta_{N_{p+1}}|\} }   \Big\} \, \frac{1}{k}\, ||\nabla _h v||_{0,\mathcal{E}_h} ^2      
 \, ,
  \end{array}
$$
for every $v\in V(\mathcal{T}_h)$. This leads us to study the  $L^2$ norm in the scatterer as well as the $L^2$ norms on the skeleton for the error. We do so  under the assumption that the exact solution satisfies the regularity and stability properties in (\ref{boundstab_Uphi}).

Under this regularity assumption, we can make use of the following best approximation properties~\cite{moiola_PhD,hiptmair+al2011VekuaTh}.  Unfortunately, these results are only proved for real $\mathrm{n}$ (there is strong numerical evidence that the plane wave expansion with evanescent components provides good approximation properties even for complex media~\cite{tsukerman}):
\begin{lemma}\label{lem:approxprop}
Assume that $\mathrm{n}$ is real,   $1\leq s\leq \frac{p-1}{2}$ and consider  $v\in V(\mathcal{T}_h)\cap H^{s+1}( \Omega_R ) $.
\begin{itemize}
\item $h$-refinement: for $p$ fixed and for each $h>0$, there is $v_h\in V_h(\mathcal{T}_h)$ such that
\begin{equation*}
\begin{array}{l}
 ||v-v_h||_{0,\Omega_i} \,\leq \, C_p\, h^{s+1} \, ||v||_{s+1,k,\Omega_R} \\[1ex]
\displaystyle ||v-v_h||_{0,\mathcal{E}_h}  
\, \leq \, C_p \, h^{s+1/2}\,    ||v||_{s+1,k,\Omega_R}  \, ,\\[1ex]
\displaystyle ||\nabla_h (v-v_h)||_{0,\mathcal{E}_h}
\,\leq\, C_p \, h^{s-1/2} \, ||v||_{s+1,k,\Omega_R}  \, ,
\end{array}
\end{equation*}
where the constant  $C_p>0$ is independent of $v$ as well as the parameters $k$, $s$, and $h$.
\item $p$-refinement: for $h$ fixed and each $p$, there is $v_h\in V_h(\mathcal{T}_h)$ such that in 2D
\begin{equation*}
\begin{array}{l}
 ||v-v_h||_{0,\Omega_i}  \,\leq \, C_h\, (\frac{log(p+2)}{p})^{s+1}\, ||v|| _{s+1,k,\Omega_R} \\[1ex]
||v-v_h||_{0,\mathcal{E}_h}  
\, \leq \, C_h \, (\frac{log(p+2)}{p})^{s+1/2} \, ||v|| _{s+1,k,\Omega_R}  \, ,\\[1ex]
||\nabla_h (v-v_h)||_{0,\mathcal{E}_h}  
\,\leq\, C_h \,  (\frac{log(p+2)}{p})^{s-1/2} \, ||v||_{s+1,k,\Omega_R}  \, ,
\end{array}
\end{equation*}
and in 3D
\begin{equation*}
\begin{array}{l}
 ||v-v_h||_{0,\Omega_i}  \,\leq \, C_h\, p^{-\ell (s+1)} \, ||v||_{s+1,k,\Omega_R} \\[1ex]
\displaystyle ||v-v_h||_{0,\mathcal{E}_h} 
\, \leq \, C_h \,  ( h^{-1/2}+1) \, p^{-\ell (s+1/2)}   \, ||v||_{s+1,k,\Omega_R}  \, ,\\[1ex]
\displaystyle ||\nabla_h (v-v_h)||_{0,\mathcal{E}_h}  
\, \leq \, C_h \,  ( h^{-1/2}+1) \, p^{-\ell (s-1/2)}   \, ||v||_{s+1,k,\Omega_R}  \, ,
\end{array}
\end{equation*}
where $\ell\equiv\ell (\mathcal{T}_h)\in (0,1]$ is $\frac{1}{\pi}$ times the maximum angle for the exterior cone condition for the mesh cells, and the constant $C_h>0$ is independent of $v$ as well as the parameters $k$, $s$ and $p$.
\end{itemize}
\end{lemma}
Therefore, under our assumptions on the exact solution, we have $u\in V(\mathcal{T}_h)\cap H^{s+1}( \Omega_R ) $ for some $1\leq s \leq \frac{p-1}{2}$, and thanks to the error bound of Lemma~\ref{lem:quasioptL2ok}, we conclude the following rates of convergence.
\begin{theorem}\label{lem:quasioptL2all} Assume that $\mathrm{n}$ is real.
The discrete solution $u_h\in V_h(\mathcal{T}_h)$ satisfies the error bound
\begin{itemize}
\item $h$-refinement: for $p$ fixed and when we refine the mesh,
$$
\begin{array}{l}
\displaystyle
 || u-u_h||_{0,\Omega_R } \, \leq\, C_p\,   h^{s-1} \, ||u_{inc}||_{s+1,k,\Omega_R}
 \, ,
\end{array}
$$
where $C_p$ is independent of $kh$ (but it depends on $p$ fixed);
\item $p$-refinement: for a fixed partition $\mathcal{T}_h$ and when we increase the order,
$$
\begin{array}{l}
\displaystyle
 || u-u_h||_{0,\Omega_R }    \leq   C_h\,  
 \left\{ \begin{array}{ll}
 \displaystyle\!\! \Big(\frac{\log (p+2)}{p}\Big)^{s-1/2}  &\mbox{if } m=2
    \\[2ex]
    \displaystyle p^{-\ell/2}  &\mbox{if } m=3 
 \end{array}\right\} \,  ||u_{inc}||_{s+1,k,\Omega_R}  
 \, ,
\end{array}
$$
where now  $C_h$ is independent of $p$ (but it depends on $kh$ fixed).
\end{itemize}
\end{theorem}
\begin{remark} Provided the results in Lemma \ref{lem:approxprop} hold also for complex $\mathrm{n}$ (c.f. \cite{tsukerman}), the above estimates also hold in that case.
\end{remark}


\section{Numerical results}\label{sec:numerics}
In this section we show numerical results for the proposed TDG method.  Unless otherwise stated we use $a=b=d_1=d_2=\frac{1}{2}$.  The experiments are all in 2D, and for a waveguide with the cross section $\Sigma=(0,H)$ where $H=1$.

All results are computed in Python using the SciPy sparse direct solver {\tt spsolve}.  We report the relative $L^2$ error in the examples defined as
\[
\frac{\Vert u - u_h\Vert_{L^2(\Omega_R)}}{\Vert u\Vert_{L^2(\Omega_R)}} \, ,
\]
where $u$ is the reference solution and $u_h$ is the TDG solution.
\subsection{Fundamental solution}
To test the predicted convergence rates, we first consider a waveguide with no obstacle and take as incident wave an approximation to the fundamental solution
with source point $\ybold=(-1.5R,0.3H)$ outside the computational domain $\Omega_R$. This example probes the behavior of the NtD map truncation condition, and the basic convergence estimates for plane waves. More precisely, we consider
\begin{equation}\label{def:fundsol_WG}
G^{N_f}_{k}(\xbold,\ybold) = -\sum_{j=0}^{N_f} \frac{e^{i\beta_j |x_1-y_1|}}{2i\beta_j} \,\theta_j(\hat{\xbold})\,\theta_j(\hat{\ybold}) \quad\mbox{for } \xbold, \ybold\in\Omega\,\,\mbox{with }x_1\neq y_1 \, ,
\end{equation}
where we set $N_f=20$ and $R=\frac{2\pi}{k}$ (one wavelength) for several values of the wave number $k$.
This gives a known smooth solution inside $\Omega_R$, 
in which all the integrals in TDG can be performed analytically. When $k=8.0$, for $j\geq 3$ in (\ref{def:fundsol_WG}) the modes comprising the fundamental solution are evanescent, but still must be modelled by the NtD map at $S_R$.   
Here we choose to truncate the NtD expansion at $M=15$ modes (the results show that this is sufficient to give a relative error below $10^{-8}$ at the finest discretization).

Results are shown in Fig.~\ref{fig:fund_outside_8} for the wave number $k=8.0$.  In the top panel we show convergence as the number of plane waves per element increases (we use the same number of plane waves for all the elements) for various different choices
of the mesh size $h$.  This log-linear graph suggests that $p$-convergence is approximately exponential on each mesh except for the finest grid at $N_p=15$.  This break down is likely due to ill-conditioning~\cite{hut03}.

In the bottom panel of Fig.~\ref{fig:fund_outside_8} we show convergence under mesh refinement as the mesh diameter $h$ decreases.  The log-log plot suggests a power law $h^s$ using $N_p$ plane waves such that $s$ increases as $N_p$ increases.  In Fig.~\ref{fig:fund_outside_8} we plot a trend line fitted to the convergence curve when $N_p=7$ which gives slightly over 4th order convergence.  Similarly we see approximately $O(h^{6.4}) $ convergence when $N_p=13$.  These results show faster rates of convergence  than predicted by the bounds in Theorem~\ref{lem:quasioptL2all}, and are
similar to the convergence rates seen in \cite[Fig.~1.3.4]{cessenat96}.

 \begin{figure}
     \centering
     \includegraphics{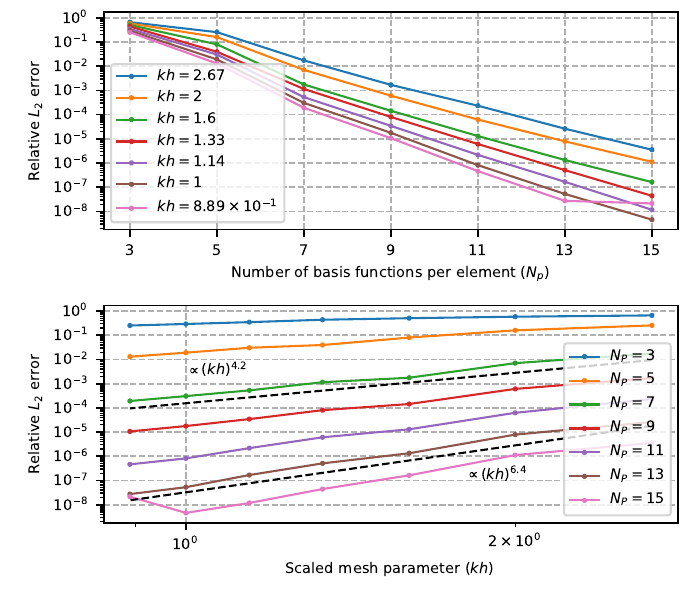}
     \caption{Convergence measured in the relative $L^2(\Omega_R)$ norm as a function of the discretization parameters.  Top panel: convergence as $N_p$, the number of plane waves per element, increases for various fixed meshes. Lower panel: convergence as the mesh diameter $h$ decreases using different numbers of directions per element. Here the approximation of the fundamental solution $G^{N_f}_k(\xbold,\ybold)$ is used as the incident field, and no scatterer is present.}
     \label{fig:fund_outside_8}
 \end{figure}

In the Fig.~\ref{fig:NTD} we study the dependence of the error in approximating the fundamental solution  on the number of modes $M$ used to approximate the NtD map. Here we consider the wave numbers $k$ between 8.0 and 32.0, and fix $R=1$, $N_p=13$ and $h=0.1$. Note that the domain size is fixed and not varied with $k$.  The error curves 
show that there is no convergence until we have included all the propagating modes in the NtD map (given by $M=\left\lfloor \frac{k\, H}{\pi}\right\rfloor$ modes, where $\lfloor \cdot \rfloor$ represent the integer part).  For higher $M$  we see convergence with a rate depending on the size of the domain in wavelengths, and the magnitude of the first purely imaginary $\beta_j$.
\begin{figure}
    \centering
    \includegraphics{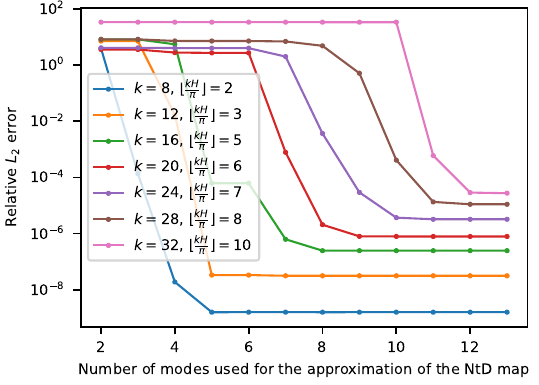}
    \caption{Relative $L^2(\Omega_R)$ error in approximating $G^{N_f}_k(\xbold,\ybold)$ as a function of the truncation parameter $M$ for the NtD map at several wave numbers.  It is necessary to include at least $\left\lfloor \frac{k\, H}{\pi}\right\rfloor$ modes (i.e. all propagating modes), and the inclusion of a few evanescent modes results in a rapid decrease in the relative error.}
    \label{fig:NTD}
\end{figure}

\subsection{Scattering from a bounded obstacle}\label{scatter}
Unlike the free space scattering problem, we do not have a simple analytic solution for scattering from an penetrable obstacle in a waveguide in order to test our method.  Therefore we compare the solution from our TDG method with a finite element solution computed on the same computational domain $\Omega_R$, but using the Perfectly Matched Layer to truncate the domain by adding an absorbing layer to the left and right of $\Omega_R$ (computed using NGSolve~\cite{netgen}).

The scatterer is a penetrable square $[-0.15,0.15]\times[0.45,0.75]$  having $\mathrm{n}=9+4i$ inside the square. The incident field corresponds to the first propagating mode with $k=8.0$, and we fix {$R=\lambda$}.  For each requested mesh size $h$, the mesh size inside the scatterer is scaled down by a factor of $1/3$ to allow for the shorter wavelength inside the scatterer (we do not scale for absorption since the absorption is quite small here).  The TDG solution and the mesh when {$kh=6.67\times 10^{-1}$  and $N_p=11$} is shown in Fig.~\ref{fig:field_absorbing}.
\begin{figure}
     \centering
     \includegraphics{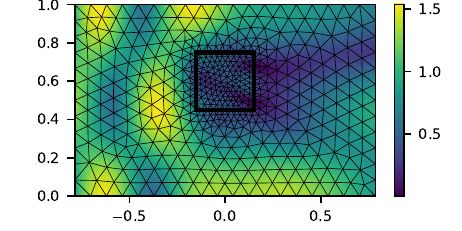}
     \caption{A density plot of the magnitude of the 
     total field computed by TDG when $kh=6.67\times 10^{-1}$ and $N_p=11$ on $\Omega_R$ for $R=1$ when the scatterer is absorbing.  The black square is the boundary of the scatterer where $\mathrm{n}=9+4i$, and the mesh is refined inside that region.}
     \label{fig:field_absorbing}
 \end{figure}
Convergence curves as the discretization parameters $kh$ and $N_p$ vary are shown in Fig.~\ref{fig:fund_outside_8a}.
As expected, in the top panel of Fig.~\ref{fig:fund_outside_8a} we see that the $p$-convergence is no longer exponential because the solution is no longer analytic, but that increasing $N_p$ brings rapid convergence until ill-conditioning starts to degrade the solution.  Convergence curves for fixed $N_p$ and varying the mesh diameter $h$ are shown in the lower panel. The $h$-convergence rate increases with $N_p$ at first but for $N_p=13$ and 15 convergence is lost on finer grids probably due to ill-conditioning. 
\begin{figure}
\centering\includegraphics{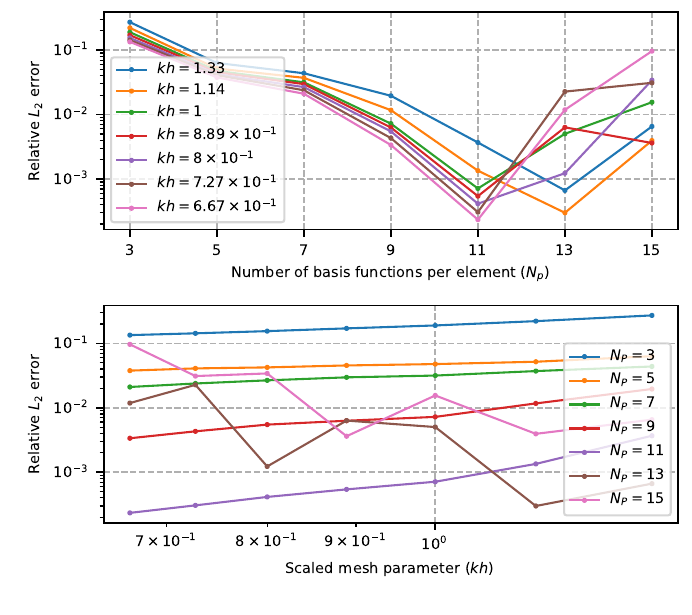}
     \caption{Convergence curves for our TDG method with the lossy penetrable scatterer described in Subsection~\ref{scatter}.  As shown in Fig.~\ref{fig:field_absorbing} the mesh inside the scatterer is refined slightly. The top panel shows the relative $L^2(\Omega_R)$ error as the number of plane waves $N_p$ is increased.  The lower panel shows the relative $L^2(\Omega_R)$ error as the mesh is refined.}
     \label{fig:fund_outside_8a}
 \end{figure}

\FloatBarrier

\subsection{Propagation through a refined mesh layer}

\begin{figure}    \centering\includegraphics{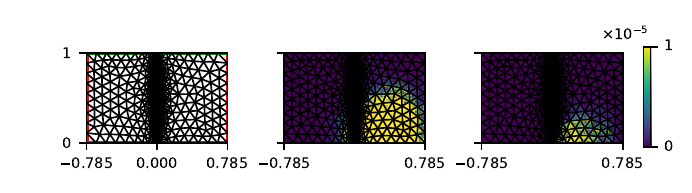}
     \caption{Left panel: The mesh when  $kh=8
     .9\times 10^{-1}$. Middle panel: a density plot of the absolute value of the error in $\Omega_R$ when $N_p=15$ and $\gamma=0$.  Right panel: a density plot of the absolute value of the error in $\Omega_R$ when $N_p=15$ and $\gamma=0.55$.}
    \label{fig:fine_mesh_error}
\end{figure}

\begin{figure}
    \centering
\resizebox{0.9\textwidth}{!}{\includegraphics{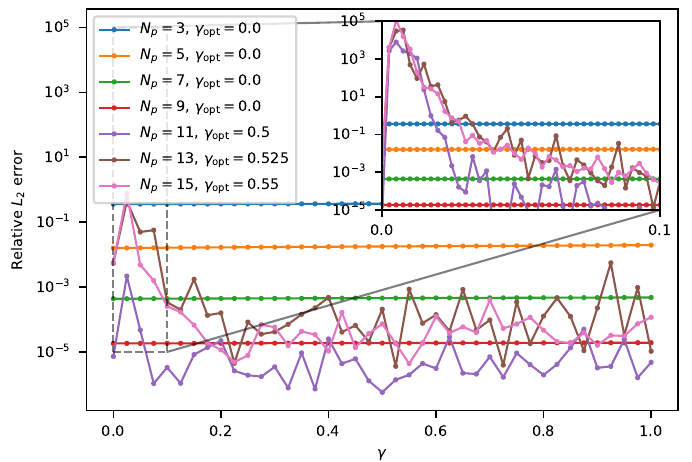}}
    \caption{The relative $L^2(\Omega_R)$ error as a function of $\gamma$ for a fixed mesh and different numbers of plane waves $N_p$ for the locally refined mesh example. Here $\gamma_{opt}$ records the value of $\gamma$ for the minimum relative error in the range $[0,1]$. For $\gamma>0$ but close to zero there is an unexplained sudden rise in the
    error.
    }
    \label{fig:cm_sweep}
\end{figure}

\begin{figure}
 \centering
\includegraphics{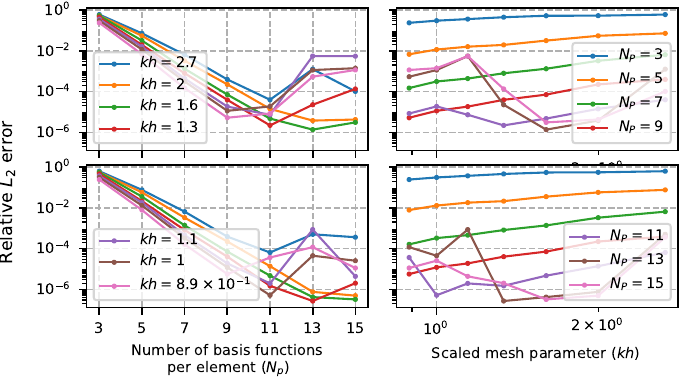}
\caption{Detailed convergence curves for the locally refined mesh example.  Top row: relative $L^2(\Omega_R)$ error as a function of $N_p$ and $h$ when $\gamma=0$. Bottom row: relative $L^2(\Omega_R)$ error as a function of $N_p$ and $h$ when $\gamma=0.55$.}
    \label{fig:fine_mesh_convergence}
\end{figure}

In \cite{hiptmair+al2014} the choice of mesh dependent parameters $a$ and $b$ in TDG methods was investigated in order to prove convergence on locally refined meshes. In this subsection we test if the use of  mesh dependent parameters $a$, $b$ and $d_1$ in our TDG method can also improve convergence in cases where too many plane waves are used on an element.

In this experiment, the waveguide does not contain any scatterer, and we seek to approximate a single propagating mode (in this case the first mode) along the waveguide in which it propagates from left to right.  However this mode has to pass through a  refined section of the grid as shown in Fig.~\ref{fig:fine_mesh_error} (left panel). In this mesh the ratio of maximum to minimum edge length is $24.6$.

Since we generally observe good approximations using the UWVF choice of parameters on quasi-uniform grids, we use, edge by edge,
\[
a\equiv a(e)=\frac{1}{2}\left(1+\gamma\left(\frac{\ell_{\rm max}}{\ell_e}-1\right)\right) \quad\mbox{and}\quad
b\equiv a(e) \qquad \mbox{on } e\in\mathcal{E}_h^I \, ,
\]
where $\gamma$ is a non-negative constant. On the boundary of $\Omega_R$, the parameter $d_1$ is chosen in the same way, but $d_2=\frac{1}{2}$. Here $\ell_{\rm max}$  is the maximum edge length in the mesh, and $\ell_e$ is the length of edge $e$.  Thus $\gamma=0$ gives the UWVF choice whereas $\gamma=1$ gives a choice of the parameter proposed in \cite{hiptmair+al2014}.

We start by choosing $k=8.0$, $R=1.0$ and $h=2.3\times10^{-1}$, and propagate the lowest order mode along the pipe.  We solve the discrete TDG problem for $\gamma\in [0,1]$, and for each choice of $\gamma$ we compute the relative $L^2(\Omega_{R})$ error in the solution compared to the exact mode.  
The results in Fig.~\ref{fig:cm_sweep} clearly show that for small numbers of plane waves, $N_p\leq 9$, the error is only slightly increasing but almost independent of $\gamma$ as $\gamma$ increases.  

For $N_p>9$ the choice of $\gamma>0$ has a profound effect.  We see an unexpected (and unexplained) rapid rise in the error by several orders of magnitude for $\gamma$ close to zero.  Increasing $\gamma$ beyond this value results in a decrease in error followed by a noisy regime.  In all cases studied here $\gamma\approx 0.5$ results in the lowest error, below that at $\gamma=0$.  However the error 
decrease is not smooth.

To investigate this further we make the choice $\gamma=0.55$ and investigate both $p$-convergence and $h$-convergence (comparing to $\gamma=0$).  In the top row of Fig.~\ref{fig:fine_mesh_convergence} we see that, for a fixed mesh size, convergence occurs as $N_p$ increases until a breakdown at $N_p\approx 9$ likely due to ill-conditioning arising from the small cells in the problem.  The choice of $\gamma=0.55$ controls the rise in error, but does not restore the convergence rate for higher values of $N_p$.

The results in the lower panel of Fig.~\ref{fig:cm_sweep} show that, for a fixed $N_p>9$, convergence ultimately breaks down as the mesh is refined. Choosing $\gamma=0.55$ helps to control the effects of this breakdown, but does not avoid it.

In Fig.~\ref{fig:fine_mesh_error} we compare a density plot of the error when $\gamma=0$ (middle panel) and $\gamma=0.55$ (right hand panel) for the locally refined mesh (left hand panel).  Although $\gamma=.55$ greatly reduces the region of maximum  error, it does not eliminate it.


\section{Conclusions and future work}\label{sec:conclusion}
We have shown how a modified form of the TDG scheme can be used to solve wave propagation problems in a waveguide with a lossy scatterer.
This approach can also be used for other applications such as scattering from a bounded object.

Our study also shows how to modify the classical TDG to handle coupling with the NtD operator.  Our convergence results  illustrate numerically  the order of convergence and demonstrate that the method works well for absorbing scatterers in a waveguide. 

Our final example investigates the use of mesh dependent parameters with the hope of improving stability and accuracy when using a locally refined grid.  The results showed almost no influence of the  mesh dependence parameter $\gamma$ for discretizations with a smaller number of directions, but mesh dependent parameters did help to control error for larger numbers of plane waves (although it did not result in continued convergence as the mesh is refined).
More investigation is needed for highly refined grids where a combination of varying $N_p$ from element to element and choosing $\gamma>0$ may improve convergence.

We have provided a convergence analysis and shown examples with a constant number of directions per element.  Our results exhibit increasing error for finer discretizations in some cases and underline the need to control ill-conditioning by varying the number of directions per element as is usually done in TDG codes.

A study of the approximation properties of plane waves when the Helmholtz equation has complex coefficients (i.e. an extension of Lemma~\ref{lem:quasioptL2all} to $\Im(n)>0$) would complement our results.


\bibliographystyle{abbrv}


\end{document}